\documentclass[12pt,letterpaper]{amsart}
\usepackage{hyperref}

\usepackage{amssymb}
\usepackage{xcolor}
\usepackage[margin=1in]{geometry}

\def\R{\mathbb R}

\def\N{\mathbb N}

\newtheorem{theorem}{Theorem}[section]
\newtheorem{lemma}[theorem]{Lemma}
\theoremstyle{definition}
\newtheorem{definition}[theorem]{Definition}

\newtheorem{corollary}[theorem]{Corollary}
\newtheorem{proposition}[theorem]{Proposition}
\theoremstyle{remark}
\newtheorem{remark}[theorem]{{\bf Remark}}

\numberwithin{equation}{section}

\begin{document}
\title[TRAVELING WAVES IN NONCLASSICAL DIFFUSION EQUATIONS WITH
  DELAY]
{TRAVELING WAVES IN NONCLASSICAL DIFFUSION EQUATIONS WITH
 DELAY}

\author  [W. Barker, L.X. Dong$^{\natural}$, V. T. Luong, N.D. Toan] {William Barker, Le Xuan Dong, Vu Trong Luong, Nguyen Duong Toan}

\address{William Barker 
 \hfill\break 
 Department of Mathematics and Statistics, University of Arkansas at Little Rock,  2801
 S University Ave, Little Rock, AR 72204. USA.}
\email{wkbarker@ualr.edu}
 
\address{Le Xuan Dong
	\hfill\break 
	Faculty of Basic Science, Viet Tri University of Industry,  9 Tien Son, Tien Cat, Viet Tri, Phu Tho, Viet Nam}
\email{donglx@vui.edu.vn}

\address{Vu Trong Luong
	\hfill\break 
	VNU-University of Education, Hanoi 144 Xuan Thuy, Cau Giay, Hanoi, Vietnam
	Vietnam National University}
\email{vutrongluong@vnu.edu.vn}

\address{Nguyen Duong Toan \hfill\break
Department of Mathematics, Haiphong University\hfill\break
171 Phan Dang Luu, Kien An, Haiphong, Vietnam}
\email{toannd@dhhp.edu.vn}

\thanks{$^{\natural}$ Corresponding author: toannd@dhhp.edu.vn}

\begin{abstract}
	This paper is concerned with the existence of monotone traveling wave solutions
	for a class of nonclassical reaction--diffusion equations with a discrete delay
	in the reaction term, namely
	\[
	\frac{\partial u(x,t)}{\partial t}
	= D\,\frac{\partial^2 u(x,t)}{\partial x^2}
	+ \alpha\,\frac{\partial^3 u(x,t)}{\partial x^2\partial t}
	+ f\big(u(x,t-\tau)\big),
	\]
	where $D>0$, $\alpha\in\mathbb{R}$, $\tau>0$, and $f$ satisfies bistable or monostable
	conditions on $[0,K]$. By applying the traveling-wave ansatz $u(x,t)=\phi(x+ct)$,
	the problem is reduced to a functional differential equation for the
	wave profile. We employ the Green's function associated with the underlying
	third-order operator and develop a monotone iteration scheme based on 
	upper and lower solutions in a weighted Banach space. Using Schauder's fixed
	point theorem, we establish the existence of monotone wavefronts connecting the
	steady states $0$ and $K$. Our approach provides a unified framework that
	extends the non-delayed nonclassical model to the delayed setting, capturing the
	effect of discrete reaction delay on the propagation of traveling waves.
\end{abstract}

\maketitle
%

 \section{Introduction}
 
 Reaction--diffusion equations with delayed effects have received considerable attention 
 in the study of wave propagation phenomena arising in biology, physics, and engineering. 
 In many physical models, classical diffusion fails to accurately capture memory or nonlocal 
 effects, leading to the introduction of nonclassical diffusion terms, such as those involving 
 mixed derivatives or higher-order temporal interactions; see 
 Aifantis~\cite{Aifantis}, Ting~\cite{Ting}, and Truesdell--Noll~\cite{Truesdell}. 
 An important example is the nonclassical reaction--diffusion equation
 with an additional mixed derivative,
 \begin{equation}\label{Nonclassical}
 	\frac{\partial u(x,t)}{\partial t}
 	= D \frac{\partial^2 u(x,t)}{\partial x^2}
 	+ \alpha \frac{\partial^3 u(x,t)}{\partial x^2 \partial t}
 	+ f\big(u_t(x)\big),
 \end{equation}
 where $D>0$, $\alpha\in\mathbb{R}$, and $f:C([-\tau,0],\mathbb{R})\to\mathbb{R}$ 
 introduces a delayed reaction term depending on the history segment
 $u_t(x)(r)=u(x,t+\theta)$ for $\theta\in[-\tau,0]$. Such delays may arise due to maturation time 
 in population dynamics or latency effects in phase transitions; 
 we refer to the classical monographs of Diekmann et al.~\cite{Diekmann} and Wu and Zou~\cite{wuzou}.
 
 Traveling wave solutions, of the form $u(x,t)=\phi(\xi)$ with $\xi=x+ct$, 
 play a crucial role in describing front propagation and transition layers 
 between stable equilibria. When \eqref{Nonclassical} is transformed into 
 the corresponding profile equation, one obtains a functional differential 
 equation of . In the absence of delay, nonclassical wave equations 
 have been studied in \cite{Bark}, where the method of upper and lower solutions 
 combined with Green’s functions was used to establish monotone wavefronts.
 Recently, equations with delays in both diffusion and reaction terms were 
 analyzed by Barker and Nguyen~\cite{BarkNguy}, and further developed in 
 \cite{BarkNguyTha1,BarkNguyTha2,Zhang2025}, highlighting the delicate role of delay 
 on wave speed and monotonicity.
 
 In the broader literature, there has been a surge of recent works focusing on 
 traveling waves in delayed or nonlocal media. Zhang, Hu, and Huang~\cite{Zhang2025} 
 investigated monotone waves for systems with delay in both diffusion and reaction. 
 Li et al.~(2024) considered a time-delayed nonlocal model arising in epidemiology, 
 while Huang (2025) analyzed $p$-Laplacian diffusion with temporal delay. 
 Zhu and Peng (2025) treated discrete-time predator--prey models with diffusion 
 and delay via an iterative fixed point framework. On the methodological side, 
 Boumenir and Minh (2006) applied Perron’s theorem and monotone iteration 
 to mixed-type functional equations, laying foundations relevant to our approach.
 
 In this work, we establish the existence of monotone traveling wave solutions 
 for \eqref{Nonclassical} by employing a Green's function representation of 
 the associated linear operator and a monotone iteration scheme based on upper 
 and lower solutions. Unlike the classical case, the presence of delay in the 
 reaction term introduces functional dependencies in the nonlinear operator, 
 leading to technical challenges in verifying invariant cones and regularity.
 We overcome these obstacles by constructing a suitable operator $F=-\mathcal{L}^{-1}H$ 
 and showing that iterates preserve order and improve regularity.
 
 The rest of the paper is organized as follows. In Section~\ref{Preliminaries and Notation}, 
 we introduce the functional framework, the Green’s function associated with the linear 
 operator, and state the hypotheses on $f$. Section~\ref{Main Results} develops the 
 monotone iteration scheme and proves the existence of traveling wave profiles 
 through the Schauder fixed point theorem. In Section~\ref{applications}, 
 we apply the abstract theory to a concrete nonclassical model with discrete delay, 
 construct explicit super- and subsolutions, and establish the existence of a 
 monotone wavefront connecting two equilibria.

 \medskip

\section{Preliminaries and Notation}\label{Preliminaries and Notation}
We adopt the following standard notational conventions throughout this paper. Let $ \mathbb{R} $ denote the field of real numbers. For any $ z \in \mathbb{R} $, we write $ \Re(z) = z $.

We denote by $ BC(\mathbb{R}, \mathbb{R}) $ the space of all bounded continuous functions $ f \colon \mathbb{R} \to \mathbb{R} $, equipped with the supremum norm
$$
\| f \| := \sup_{t \in \mathbb{R}} | f(t) |.
$$
Similarly, $ BC^k(\mathbb{R}, \mathbb{R}) $ denotes the subspace of $ C^k(\mathbb{R}, \mathbb{R}) $ consisting of functions whose derivatives up to order $ k $ are bounded. When boundedness is not assumed, we write $ C(\mathbb{R}, \mathbb{R}) $ or $ C^k(\mathbb{R}, \mathbb{R}) $.

A natural partial order on $ BC(\mathbb{R}, \mathbb{R}) $ is defined by
$$
f \le g \quad \Longleftrightarrow \quad f(t) \le g(t) \text{ for all } t \in \mathbb{R},
$$
and
$$
f < g \quad \Longleftrightarrow \quad f(t) \le g(t) \text{ for all } t \in \mathbb{R}, \text{ with } f(t) \ne g(t) \text{ for all } t \in \mathbb{R}.
$$
An operator $ P $ acting on a subset $ S \subset BC(\mathbb{R}, \mathbb{R}) $ is said to be \emph{monotone} if $ f \le g $ implies $ Pf \le Pg $ for all $ f, g \in S $. In particular, if $ P $ is linear and acts on all of $ BC(\mathbb{R}, \mathbb{R}) $, it is monotone if and only if it maps nonnegative functions to nonnegative functions.

For any $ \alpha \in \mathbb{R} $, we denote by $ \hat{\alpha} $ the constant function defined by $ \hat{\alpha}(t) = \alpha $ for all $ t \in \mathbb{R} $.

We now recall standard definitions of function spaces used throughout this work.

\begin{definition}[$ L^p $ Spaces]
Let $ n \in \mathbb{N} $. For $ 1 \le p \le \infty $, the space $ L^p(\mathbb{R}, \mathbb{R}) $ consists of all measurable functions $ f \colon \mathbb{R} \to \mathbb{R} $ such that
$$
\| f \|_{L^p} =
\begin{cases}
\left( \displaystyle\int_{\mathbb{R}} |f(x)|^p \, dx \right)^{1/p}, & 1 \le p < \infty, \\[6pt]
\displaystyle \operatorname{ess\,sup}_{x \in \mathbb{R}} |f(x)|, & p = \infty,
\end{cases}
$$
is finite.
\end{definition}

\begin{definition}[Sobolev Space $ W^{1,p} $]
For $ 1 \le p \le \infty $, the Sobolev space $ W^{1,p}(\mathbb{R}, \mathbb{R}) $ is defined by
$$
W^{1,p}(\mathbb{R}, \mathbb{R}) = \left\{ f \in L^p(\mathbb{R}, \mathbb{R}) \;\middle|\; f \text{ is absolutely continuous, and } f' \in L^p(\mathbb{R}, \mathbb{R}) \right\}.
$$
In one dimension, this space is continuously embedded in $ L^\infty(\mathbb{R}, \mathbb{R}) $; that is, there exists a constant $ C > 0 $ such that
$$
\| f \|_{L^\infty} \le C \| f \|_{W^{1,p}} \quad \text{for all } f \in W^{1,p}(\mathbb{R}, \mathbb{R}).
$$
\end{definition}

\begin{definition}[Sobolev Space $ W^{k,\infty} $]
We define
$$
W^{2,\infty}(\mathbb{R}, \mathbb{R})
=
\left\{
f \in L^\infty(\mathbb{R}, \mathbb{R})
\;\middle|\;
f' \in L^\infty(\mathbb{R}, \mathbb{R}),\,
f' \text{ is absolutely continuous, and } f'' \in L^\infty(\mathbb{R}, \mathbb{R})
\right\}.
$$
Equivalently,
$$
W^{2,\infty}(\mathbb{R}, \mathbb{R})
=
\left\{ f \in W^{1,\infty}(\mathbb{R}, \mathbb{R}) \;\middle|\; f' \in W^{1,\infty}(\mathbb{R}, \mathbb{R}) \right\}.
$$
\end{definition}

\begin{definition}[Sobolev Space $ W^{3,\infty} $]
The space $ W^{3,\infty}(\mathbb{R}, \mathbb{R}) $ consists of all functions $ f \in W^{2,\infty}(\mathbb{R}, \mathbb{R}) $ such that the second derivative $ f'' $ is absolutely continuous and the third derivative $ f''' \in L^\infty(\mathbb{R}, \mathbb{R}) $. That is,
$$
W^{3,\infty}(\mathbb{R}, \mathbb{R}) =
\left\{
f \in W^{2,\infty}(\mathbb{R}, \mathbb{R})
\;\middle|\;
f'' \text{ is absolutely continuous, and } f''' \in L^\infty(\mathbb{R}, \mathbb{R})
\right\}.
$$
Equivalently,
$$
W^{3,\infty}(\mathbb{R}, \mathbb{R}) =
\left\{ f \in W^{2,\infty}(\mathbb{R}, \mathbb{R}) \;\middle|\; f'' \in W^{1,\infty}(\mathbb{R}, \mathbb{R}) \right\}.
$$
\end{definition}

 We now recall some important theorems that will be used in the subsequent sections.
 \subsection*{2.2. Rouché's Theorem}
 Let $A$ be an open subset of $\mathbb{C}$, $f$ and $g$ are two analytic functions on $A$. 
 A piecewise continuously differentiable function $\gamma : [a,b] \to \mathbb{C}$ such that 
 $\gamma(a) = \gamma(b)$ is called a circuit. The following theorem (\cite{7}, Rouché's Theorem, p.~247) 
 will be used later on:
 
 \begin{theorem} \label{Rouche}
 	Let $A \subset \mathbb{C}$ be a simply connected domain, $f, g$ two analytic complex valued 
 	functions in $A$. Let $T$ be the (at most denumerable) set of zeros of $f$, 
 	$T'$ the set of zeros of $f+g$ in $A$, $\gamma$ a circuit in $A - T$, defined on an interval $I$. 
 	Then, if $|g(z)| < |f(z)|$ in $\gamma(I)$, the function $f+g$ has no zeros on $\gamma(I)$, and
 	\begin{equation}\label{2.1}
 		\sum_{a \in T} j(a;\gamma)\,\omega(a;f) \;=\; 
 		\sum_{b \in T'} j(b;\gamma)\,\omega(b; f+g),
 	\end{equation}
 	where $j(a;\gamma)$ is the index of $\gamma$ with respect to $a$, and $\omega(a,f)$ is the multiplicity 
 	of the zeros of $f$ at $a$.
 \end{theorem}
 
 As a consequence of the theorem, if $\gamma$ is a simple closed circuit that encircles $T$, then 
 $j(a,\gamma)=1$ for all $a$ inside $\gamma(I)$ and Rouché Theorem claims that the total zeros 
 (counting multiplicities) of $f$ and $f+g$ inside the circuit $\gamma$ are the same.

 
%
%
%

 
 \section{Main Results}\label{Main Results}

 We consider traveling waves for a non classical reaction diffusion equation 
 
\begin{equation}\label{Nonclassical}
	\frac{\partial u(x,t)}{\partial t}
	= D \frac{\partial^2 u(x,t)}{\partial x^2}
	+ \alpha \frac{\partial^3 u(x,t)}{\partial x^2 \partial t}
	+ f\big(u_t(x)\big).
\end{equation}
where $D>0$, $\alpha\in\mathbb R$   and
$f:C([-\tau,0],\mathbb R)\to\mathbb R$ is continuous and for any fixed $x \in \mathbb{R}$, $u_t(x) \in C([-\tau, 0], \mathbb{R}^n)$ is defined by  
$
u_t(x)(\theta) = u(t+\theta, x), \quad \theta \in [-\tau, 0].
$

 We are interested in traveling wave solutions of the form $u(x,t) = \phi(x+ct)$, $c>0$. 
 This transformation gives
 $$
 u(x,t) = \phi(x+ct),
\qquad
 \frac{\partial u(x,t)}{\partial t} = c\phi'(x+ct),
 $$
 $$
 \frac{\partial^2 u(x,t)}{\partial x^2} = \phi''(x+ct).
\qquad
 	 \frac{\partial^3 u(x, t )}{\partial x^2\partial t} = c\phi'''(x+ct).
 $$
 
 Setting $\xi = x+ct, r = c\tau$, Eq.\eqref{Nonclassical} becomes the following ordinary differential equation
 \begin{equation} \label{5.2}
 	\alpha c\,\phi'''(\xi) + D\,\phi''(\xi) -c\,\phi'(\xi)  +   f_c(\phi_\xi) = 0,
 \end{equation}
 where $f_c \in \mathcal{X}_c$, $ \mathcal{X}_c := C([-\tau,0], \mathbb{R}) \to \mathbb{R}$, defined as
 $$
 f_c(\psi) = f(\psi^c), \quad \psi^c(\theta) := \psi(c\theta), \quad \theta \in [-\tau,0].
 $$
 
 By a shift of variable we can reduce the equation to
 \begin{equation}\label{5.3}
 		\alpha c\,\phi'''(\xi) + D\,\phi''(\xi) -c\,\phi'(\xi)    + f_c(\phi_{\xi}) = 0,
 \end{equation}
 
 The main purpose of this section is to look for solutions $\phi$ of \eqref{5.3} in the following subset of $C(\mathbb{R},\mathbb{R})$
 $$
\Gamma := \Big\{\varphi \in C(\mathbb{R},\mathbb{R}) : \varphi \ \text{is nondecreasing},\ 
\lim_{\xi \to -\infty}\varphi(\xi) = \hat 0,\ 
\lim_{\xi \to +\infty}\varphi(\xi) = \hat K \Big\}.
$$

Taking $\beta>0$, from \eqref{5.3} we can rewrite it as
\begin{equation}\label{wave1}
	\alpha c\,\phi'''(\xi) + D\,\phi''(\xi) - c\,\phi'(\xi) - \beta \phi(\xi) + H(\phi)(\xi) = 0,
\end{equation}
where
$$
H(\phi)(\xi) := \beta \phi(\xi) + f_c(\phi_\xi), \qquad \phi \in C(\mathbb{R},\mathbb{R}).
$$
As proved in \cite{wuzou}, the operator $H$ enjoys similar properties:
\begin{lemma} 
	Assume that \eqref{5.4}. Then, for any $\phi \in \Gamma$, we have that
	\begin{itemize}
		\item[(i)] $ 0 \leq  H(\phi)(\xi) \leq \beta K$, $\xi \in \mathbb{R}$,
		\item[(ii)] $H(\phi)(\xi)$ is nondecreasing in $\xi \in \mathbb{R}$,
		\item[(iii)] $H(\psi)(\xi) \leq H(\phi)(\xi)$ for all $\xi \in \mathbb{R}$, if $\psi \in C(\mathbb{R},\mathbb{R})$ is given so that $0 \leq \psi(\xi) \leq \phi(\xi) \leq K$ for all $\xi \in \mathbb{R}$.
	\end{itemize}
\end{lemma}

We impose the following conditions on $f:\, C([-\tau,0],\mathbb{R})\to\mathbb{R}$:
\begin{itemize}
	\item[(C1)]  
	There exists a constant $K>0$ such that
	\[
	f(\hat{0}) = f(\hat{K}) = 0, \qquad 
	f(\hat{u}) > 0 \ \text{for } \hat{u}\in(0,\hat{K}),
	\]
	where $\hat{u}$ denotes the constant history function $\hat{u}(s)\equiv u$ on $[-\tau,0]$.
	
	\item[(C2)]  
	There exists a constant $L>0$ such that
	\[
	|f(\phi) - f(\psi)| \le L \|\phi - \psi\|_\infty,
	\quad \forall\, \phi,\psi\in C([-\tau,0],\mathbb{R}).
	\]
	Moreover, there exists $\beta>0$ such that
	\begin{equation}\label{5.4}
		f_c(\phi) - f_c(\psi) + \beta(\phi(0) - \psi(0)) \ge 0,
	\end{equation}
	for all $\phi,\psi \in C([-\tau,0],\mathbb{R})$ satisfying 
	$0 \le \phi(s) \le \psi(s) \le K$ for all $s\in[-\tau,0]$.
	
	\item[(C3)]  
	The operator $H(\phi):=f_c(\phi)+\beta\phi$ is continuous from 
	$BC(\mathbb{R},[0,K])$ into $BC(\mathbb{R},\mathbb{R})$, and 
	\[
	\sup_{\phi\in\Gamma}\|H(\phi)\|_\infty < \infty.
	\]
\end{itemize}

 \begin{remark}
 	The condition (C1) guarantees there are only two roots or steady states for $f(u).$ In fact, for standard logistic response $\hat{K}=\hat{1}.$ 
 \end{remark}

For convenience, let $\alpha=1$,  
we consider the characteristic function associated with the linearized profile operator at delay $r$,
\begin{equation} \label{dt}
\Delta (\lambda)=c\lambda^{3}+D\lambda^{2}-c\lambda-\beta 
  \qquad \lambda\in\mathbb{C}.
\end{equation}

Throughout this section we fix $c>0$, $D>0$, and $\beta>0$. 
 \begin{proposition} \label{r0}
	For $c>0$, $D>0$, and $\beta>0$, one has $\Delta(i\omega)\neq 0$ for all $\omega\in\mathbb{R}$.
\end{proposition}

\begin{proof}
	Substituting $\lambda=i\omega$ into \eqref{dt} gives
	$$
	\Delta(i\omega)
	= c(i\omega)^3 + D(i\omega)^2 - c(i\omega) - \beta
	= \big(-D\omega^2 - \beta\big) \;+\; i\big(-c\omega^3 - c\omega\big).
	$$
	Hence $\Re \Delta(i\omega) = -D\omega^2-\beta < 0$ for every $\omega\in\mathbb{R}$, and in particular cannot vanish.
	Therefore $\Delta(i\omega)\neq 0$ for all $\omega\in\mathbb{R}$.
\end{proof}

\begin{proposition}\label{prop:rootcount}
	Let $c>0$, $D>0$, and $\beta>0$. Then:
	\begin{enumerate}
		\item[(i)] $\Delta$ has exactly one positive real root $\lambda_+>0$.
		\item[(ii)] The remaining two roots lie in the open left half-plane: either both are real and negative, or they form a complex conjugate pair with negative real part.
		\item[(iii)] Consequently, there exists $\mu_0>0$ such that every root $\lambda$ of $\Delta$ satisfies $|\Re\lambda|\ge \mu_0$, i.e.\ $\Delta$ is hyperbolic.
	\end{enumerate}
\end{proposition}

\begin{proof}
	(i) By Descartes' rule of signs applied to $\Delta(\lambda)=c\lambda^3+D\lambda^2-c\lambda-\beta$ (whose coefficients have the sign pattern $+,+,-,-$), the number of positive real roots counted with multiplicity is exactly one. Since $\Delta(0)=-\beta<0$ and $\Delta(\lambda)\to+\infty$ as $\lambda\to+\infty$, there exists a unique $\lambda_+>0$ with $\Delta(\lambda_+)=0$.
	
	(ii) Divide \eqref{dt} by $c$ to obtain the monic cubic
	$$
	\lambda^3 + \frac{D}{c}\lambda^2 - \lambda - \frac{\beta}{c}=0.
	$$
	Let $\{\rho_1,\rho_2,\rho_3\}$ be its roots (so $\{\rho_j\}$ are the roots of $\Delta$ as well). Vieta's formulas yield
	$$
	\rho_1+\rho_2+\rho_3 = -\frac{D}{c} < 0, 
	\qquad 
	\rho_1\rho_2\rho_3 = \frac{\beta}{c} > 0 .
	$$
	From (1) one root, say $\rho_1=\lambda_+>0$, is positive and real. Hence
	$$
	\rho_2+\rho_3 = -\frac{D}{c} - \lambda_+ < 0.
	$$
	If $\rho_2,\rho_3\in\mathbb{R}$ then both must be negative; if they are a complex conjugate pair, write $\rho_{2,3}=a\pm i b$ with $a=\frac{1}{2}(\rho_2+\rho_3)=\tfrac12(-\frac{D}{c}-\lambda_+)<0$, showing $\Re\rho_{2,3}<0$.
	
	(iii) Proposition~\ref{r0} ensures that no root lies on the imaginary axis. Together with (1)–(2), the set of roots is separated from $i\mathbb{R}$ by a positive distance. Define
	$$
	\mu_0 := \min\big\{\lambda_+,\, -\Re\rho_2,\, -\Re\rho_3\big\}>0.
	$$
	Then every root $\lambda$ satisfies either $\Re\lambda\ge \lambda_+\ge \mu_0$ or $\Re\lambda\le -\mu_0$, establishing hyperbolicity and a spectral gap of width at least $\mu_0$ about the imaginary axis.
\end{proof}

 Let us consider the linear operator 
 $\mathcal{L}$ in $BC(\mathbb{R}, \mathbb{R})$ defined as
 \begin{equation}\label{5.5}
 	\mathcal{L}(\phi)(\xi) :=  	\alpha c\,\phi'''(\xi) + D\,\phi''(\xi) -c\,\phi'(\xi) - \beta\,\phi(\xi)     = f(\xi), 
 \end{equation}
 where $f \in BC(\mathbb{R}, \mathbb{R})$, and

 By Propositions~\ref{r0} and \ref{prop:rootcount}, and by a standard Rouché-continuation argument for the delayed characteristic function (see e.g.\ \cite{Diekmann,ma}), the spectrum of the linearization is separated from the imaginary axis; consequently, for any $0<\mu<\mu_0$ the operator
 \[
 \mathcal L:BC^3_\mu(\mathbb R)\to BC_\mu(\mathbb R)
 \]
 is bijective and admits a bounded inverse. In particular $-\mathcal L^{-1}$ is order-preserving provided its Green kernel satisfies $-G\ge0$ (cf.\ Theorem~4.1 of \cite{wuzou}).

 With this reformulation, Eq.~\eqref{5.3} can be written as
 \begin{equation}\label{5.7}
 	\phi = -\mathcal{L}^{-1} H(\phi), \qquad \phi \in \Gamma.
 \end{equation}
 Thus, any solution $\phi$ corresponds to a fixed point of the operator
 $-\mathcal{L}^{-1}H$ on $\Gamma$. In what follows, we establish conditions
 under which this operator is well defined on $\Gamma$, or on a suitable
 closed convex subset of $\Gamma$, and acts as a compact operator there.
 Once these properties are verified, the existence of a fixed point---and
 hence of a traveling wave solution---follows directly from the
 Schauder Fixed Point Theorem.
 
 \begin{definition} \label{df5.2}
 	A function $\phi \in BC^3(\mathbb{R},\mathbb{R})$ is called an upper solution 
 	(lower solution, respectively) for the wave equation (5.3) if it satisfies the following
  	 \begin{equation*} 
 		\alpha c\,\phi'''(\xi) + D\,\phi''(\xi) -c\,\phi'(\xi)    + f_c(\phi_{\xi}) \leq 0,
 	\end{equation*}
 	\[
 	\big(\alpha c\,\phi'''(\xi ) + D\,\phi''(\xi) -c\,\phi'(\xi)    + f_c(\phi_{\xi}) \geq 0, \ \text{respectively}\big)
 	\]
 	for all $t \in \mathbb{R}$.
 \end{definition}

 
 \begin{lemma}\label{lem:upper-lower-simple}
  	Let $F:=-\mathcal L^{-1}H$ and assume (C1)--(C3). 
 	Then $\phi\in BC^3(\mathbb R,[0,K])$ is an upper solution (resp.\ lower solution) 
 	of Eq.~\eqref{5.3} if and only if $F\phi \le \phi$ (resp.\ $F\phi \ge \phi$).
 \end{lemma}

 \begin{proof}
 	We prove the upper–solution case; the lower case is identical with inequalities reversed.
 	
 	Assume $\mathcal L\phi+H(\phi)\le0$. Applying the bounded, order--preserving map $-\mathcal L^{-1}$ to both sides yields
 	\[
 	(-\mathcal L^{-1})(\mathcal L\phi + H(\phi)) \le 0.
 	\]
 	Since $-\mathcal L^{-1}\mathcal L = -\mathrm{Id}$ on $\mathrm{Dom}(\mathcal L)$,
 	this is equivalent to
 	\[
 	-\phi - \mathcal L^{-1}H(\phi) \le 0,
 	\]
 	hence $\phi \ge -\mathcal L^{-1}H(\phi)=F\phi$.
 	
 	Conversely, if $F\phi\le\phi$ then $-\mathcal L^{-1}H(\phi)\le\phi$, so
 	$-\phi-\mathcal L^{-1}H(\phi)\le0$. Applying $\mathcal L$ (valid on
 	$\mathrm{Dom}(\mathcal L)$) and using linearity gives $\mathcal L\phi+H(\phi)\le0$,
 	i.e.\ $\phi$ is an upper solution.  
 \end{proof}


The traveling-wave solution to
\begin{equation} \label{s1}
	c\,\phi'''(\xi) + D\,\phi''(\xi) - c\,\phi'(\xi) - \beta\,\phi(\xi)
	= -H\bigl(\phi(\xi)\bigr)
\end{equation}
is given by
\[
\phi(\xi)=-\int_{-\infty}^{\infty} G(\xi - y)\,H\bigl(\phi(y)\bigr)\,dy,
\]
where $G$ denotes the Green's function of the operator $\mathcal L$. 
That is, $G$ is the fundamental solution satisfying 
\[
\mathcal L G = \delta,
\]
with $\delta$ the Dirac distribution at the origin. 
Equivalently, $G$ is uniquely determined by continuity of $G$ and $G'$, and a jump condition in $G''$ at $\xi=0$ corresponding to the unit point source. 
Hence $G$ is characterized by the following conditions:
 \begin{enumerate}
 	\item[{(M1)}] 
 	$$
 	\lim_{\xi \to 0^-} G(\xi) = \lim_{\xi \to 0^+} G(\xi)  \quad \text{(continuity of } G),
 	$$
 	
 	\item[{(M2)}] 
 	$$
 	\lim_{\xi \to 0^-} G'(\xi) = \lim_{\xi \to 0^+} G'(\xi) \quad \text{(continuity of the first derivative),}
 	$$
 	
 	\item[{(M3)}] 
 	$$
 	\lim_{\xi \to 0^+} G''(\xi) - \lim_{\xi \to 0^-} G''(\xi) = \frac{1}{c} \quad \text{(jump of the second derivative induced by } \mathcal L G=\delta). 
 	$$
 \end{enumerate}

  \begin{lemma}\label{Green}
  	Let $G(\xi)$ be the Green's function of $\mathcal L$ satisfying the jump/continuity
  	conditions \textup{(M1)--(M3)}.
  	Assume that the characteristic polynomial
  	\[
  	\Delta(\lambda)=c\lambda^3 + D\lambda^2 - c\lambda - \beta
  	\]
  	has three distinct real roots $\lambda_1<\lambda_2<0<\lambda_3$ and that $c>0$.
  	Then $G$ admits the representation
  	\[
  	G(\xi)=
  	\begin{cases}
  		A_3 e^{\lambda_3\xi}, & \xi<0,\\[4pt]
  		- \bigl(A_1 e^{\lambda_1\xi}+A_2 e^{\lambda_2\xi}\bigr), & \xi>0,
  	\end{cases}
  	\]
  	where $A_1,A_2,A_3$ are determined by (M1)--(M3). Moreover,
  	\[
  	G(\xi)<0 \quad \text{for all }\xi\in\mathbb{R}.
  	\]
  \end{lemma}
  
  \begin{proof}
  	Solving the matching conditions (M1)--(M3) yields
  	\[
  	A_1=\frac{-1}{c(\lambda_1-\lambda_2)(\lambda_1-\lambda_3)},\qquad
  	A_2=\frac{-1}{c(\lambda_2-\lambda_1)(\lambda_2-\lambda_3)},\qquad
  	A_3=\frac{-1}{c(\lambda_3-\lambda_1)(\lambda_3-\lambda_2)}.
  	\]
  	Since $c>0$ and $\lambda_1<\lambda_2<0<\lambda_3$, it follows that
  	$A_1<0$, $A_2>0$, and $A_3<0$.
  	
  	For $\xi<0$, we have $G(\xi)=A_3 e^{\lambda_3\xi}<0$ because $A_3<0$ and $e^{\lambda_3\xi}>0$.
 On the other hand,  
    	for $\xi>0$, 
  	\begin{align*}
G(\xi) &\, = - \bigl(A_1 e^{\lambda_1\xi}+A_2 e^{\lambda_2\xi}\bigr)\\
  		&\, =  -e^{\lambda_2\xi}[A_1 e^{(\lambda_1-\lambda_2)\xi}+A_2]\\
  		& \, =-e^{\lambda_2\xi}T(\xi),		
  	\end{align*}
  	where $T(\xi):=A_1 e^{(\lambda_1-\lambda_2)\xi}+A_2$.
  	We have $T'(\xi)=A_1(\lambda_1-\lambda_2)e^{(\lambda_1-\lambda_2)\xi}>0$, so $T$ is increasing.
  	The continuity of $G$ at $\xi=0$ gives $A_3=-(A_1+A_2)$, hence
  	$T(0)=A_1+A_2=-A_3>0$ and therefore $T(\xi)>0$ for all $\xi>0$.
  	Thus $G(\xi)=-e^{\lambda_2\xi}T(\xi)<0$.
  	
  	Consequently, $G(\xi)<0$ for all $\xi\in\mathbb{R}$.
  \end{proof}

 \subsection{Monotone Iteration method}
  
 Let
 \[
 \Gamma := \Big\{ \phi \in BC_{[0,1]}(\mathbb{R}) :
 \ \phi\ \text{nondecreasing},\ 
 \lim_{\xi\to-\infty}\phi(\xi)=0,\ 
 \lim_{\xi\to+\infty}\phi(\xi)=1\Big\},
 \]
 and fix $0<\mu<\mu_0$. Denote
 \[
 B_\mu(\mathbb R, \mathbb R)=\{\phi\in C(\mathbb R):\|\phi\|_\mu:=\sup_{\xi\in\mathbb R} e^{-\mu|\xi|}|\phi(\xi)|<\infty\}.
 \]
 Set $F(\phi):=-\mathcal L^{-1}H(\phi)=-G*H(\phi)$, where $G$ is the Green kernel from Lemma~\ref{Green}.
 \begin{lemma}\label{Banach}
 	With the above notation:
 	\begin{enumerate}
 		\item $\Gamma\neq\varnothing$.
 		\item $\Gamma$ is closed, convex and bounded in $B_\mu(\mathbb R, \mathbb R)$.
 		\item $\Gamma\subset B_\mu(\mathbb R, \mathbb R)$.
 		\item $(B_\mu(\mathbb R, \mathbb R),\|\cdot\|_\mu)$ is a Banach space.
 	\end{enumerate}
 \end{lemma}
 \begin{proof}[Sketch]
 	(i)–(iii) are immediate (e.g.\ $\phi(\xi)=(1+\tanh\xi)/2\in\Gamma$). (iv) completeness of $BC_\mu$ is standard.
 \end{proof}

 \begin{lemma}\label{Sch1}
 	Assume (C1)--(C3) and the Green kernel properties: $G\in L^1$ with $|G(\xi)|\le C e^{-\gamma|\xi|}$ for some $\gamma>\mu$, and $-G\ge0$. Then:
 	\begin{enumerate}
 		\item $F(\Gamma)\subset\Gamma$.
 		\item $F$ is continuous on $(B_\mu(\mathbb R, \mathbb R),\|\cdot\|_\mu)$.
 		\item $F$ is compact on $\Gamma$.
 	\end{enumerate}
 \end{lemma}
%
%
%
%

 
%
%
%
 
 \begin{proof}
 	Looking at $i.)$ we first show $ \underline{\phi}\le F(\underline{\phi})$. Using \eqref{s1}, we have
 	\begin{align*}
 		 F(\underline{\phi})(t)&\,=-\int_{-\infty}^{\infty} G(t-s)H (\underline{\phi})(s) ds\\
 		&\ge \int_{-\infty}^{\infty} G(t-s)\left(c\underline{\phi}'''(s)+D\underline{\phi}''(s)-c\underline{\phi}'(s)-\beta \underline{\phi}(s)\right) ds=\underline{\phi}(t)
 	\end{align*}
 	for all $t\in \R.$
 	Next, we show $F(\overline{\phi})   \le\overline{\phi}$.
 	\begin{align*}
 		F(\overline{\phi}) (t)&\,=-\int_{-\infty}^{\infty} G(t-s)H (\overline{\phi})(s) ds\\
 		&\le \int_{-\infty}^{\infty} G(t-s)\left(c\overline{\phi}'''(s)+D\overline{\phi}''(s)-c\overline{\phi}'(s)-\beta\overline{\phi}(s)\right) ds=\overline{\phi}(t)
 	\end{align*}
 	for all $t\in \R.$  
	Thus, $$\underline{\phi} \le F(\underline{\phi}) \le F(\phi) \le F(\overline{\phi})   \le \overline{\phi}.$$
 	On the other hand, we have  
 		\begin{align*}
 			\lim_{\xi \to -\infty} \underline{\phi}(\xi) = 0, \quad 
 	\lim_{\xi \to -\infty} \overline{\phi}(\xi) = 0 \text{ then } \lim_{\xi \to -\infty} F( \phi (\xi)) = 0.
 	\\
 		\lim_{\xi \to +\infty} \underline{\phi}(\xi) = 1, \quad 
 	\lim_{\xi \to +\infty} \overline{\phi}(\xi) = 1 \text{ then } \lim_{\xi \to +\infty} F( \phi (\xi)) = 1.
 	 	\end{align*}
 	 
 	Therefore,  	 $F:\Gamma\to \Gamma.$
 	
 	Now, looking at $ii.)$ we first show $H$ is continuous with respect to  $|\cdot|_{\mu}$ in $B_{\mu}(\R,\R).$ 
 	Assume (C2): $f:C([-\tau,0],\mathbb R)\to\mathbb R$ is Lipschitz with constant $L>0$,
 	and recall that $H(\phi)(t)=f_c(\phi_t)+\beta\phi(t)$.
 	Fix $\varepsilon>0$. 
 	
 	For $\phi,\psi\in B_\mu(\mathbb R,\mathbb R)$, we have for each $t\in\mathbb R$,
 	\[
 	|H(\phi)(t)-H(\psi)(t)|
 	\le |f_c(\phi_t)-f_c(\psi_t)| + \beta|\phi(t)-\psi(t)|.
 	\]
 	Since $f_c$ is $L$–Lipschitz on histories,
 	\[
 	|f_c(\phi_t)-f_c(\psi_t)|
 	\le L\sup_{\theta\in[-\tau,0]}|\phi(t+c\theta)-\psi(t+c\theta)|.
 	\]
 	Multiplying both sides by $e^{-\mu|t|}$ and noting that 
 	for $\theta\in[-\tau,0]$ we have $|t+c\theta|-|t|\le |c\theta|=r$, we get
 	\[
 	e^{-\mu|t|}|f_c(\phi_t)-f_c(\psi_t)|
 	\le L e^{\mu r}\|\phi-\psi\|_\mu.
 	\]
 	Similarly,
 	\[
 	e^{-\mu|t|}\beta|\phi(t)-\psi(t)| \le \beta\|\phi-\psi\|_\mu.
 	\]
 	Taking the supremum over $t\in\mathbb R$ gives
 \begin{equation} \label{h est}
 	\|H(\phi)-H(\psi)\|_\mu
 \le (L e^{\mu r} + \beta)\,\|\phi-\psi\|_\mu.
 \end{equation}
 
 	Hence $H$ is globally Lipschitz (and therefore continuous) on 
 	$B_\mu(\mathbb R,\mathbb R)$ with constant $L e^{\mu r}+\beta$.
 	
 	We recall that $B_\mu(\mathbb R,\mathbb R)$ denotes the weighted Banach space 
 	equipped with the exponential norm 
 	\[
 	\|u\|_\mu:=\sup_{t\in\mathbb R} e^{-\mu|t|}|u(t)|,
 	\qquad 0<\mu<\gamma,
 	\]
 	where $\gamma$ is such that the Green kernel $G$ of $\mathcal L$ satisfies
 	$|G(\xi)|\le C_G e^{-\gamma|\xi|}$.

 	We first justify the elementary weighted inequality used below.
 	Let $h\in B_\mu(\mathbb R,\mathbb R)$. By the definition of the weighted norm,
 	for every $s\in\mathbb R$,
 	\[
 	|h(s)| = e^{\mu|s|}\bigl(e^{-\mu|s|}|h(s)|\bigr) \le e^{\mu|s|}\|h\|_\mu,
 	\]
 	which is precisely
 \begin{equation} \label{*}	|h(s)| \le \|h\|_\mu\,e^{\mu|s|}.
 \end{equation}
 
 	Applying \eqref{*} to $h=H(\phi)-H(\psi)$ gives, for all $s$,
 	\[
 	|H(\phi)(s)-H(\psi)(s)| \le \|H(\phi)-H(\psi)\|_\mu\,e^{\mu|s|}.
 	\]
 	
 	Now fix $\phi,\psi\in B_\mu(\mathbb R,\mathbb R)$. For each $t\in\mathbb R$,
 	\[
 	\begin{aligned}
 		|F(\phi)(t)-F(\psi)(t)|
 		&\le \int_{\mathbb R} |G(t-s)|\,|H(\phi)(s)-H(\psi)(s)|\,ds\\
 		&\le \|H(\phi)-H(\psi)\|_\mu
 		\int_{\mathbb R} |G(t-s)|\,e^{\mu|s|}\,ds.
 	\end{aligned}
 	\]
 	Make the change of variables \(r=t-s\). Using the inequality \(|s|=|t-r|\le |t|+|r|\) we obtain
 	\[
 	\int_{\mathbb R} |G(t-s)|\,e^{\mu|s|}\,ds
 	= \int_{\mathbb R} |G(r)|\,e^{\mu|t-r|}\,dr
 	\le e^{\mu|t|}\int_{\mathbb R} |G(r)|\,e^{\mu|r|}\,du.
 	\]
 	Set
 	\[
 	C_\mu:=\int_{\mathbb R} |G(r)|\,e^{\mu|r|}\,dr <\infty,
 	\]
 	which is finite because $G$ decays exponentially faster than $e^{-\mu|\cdot|}$ by choice of $\mu<\gamma$.
 	Combining the estimates yields, for every $t$,
 	\[
 	|F(\phi)(t)-F(\psi)(t)| \le C_\mu\,e^{\mu|t|}\,\|H(\phi)-H(\psi)\|_\mu.
 	\]
 	Multiplying by $e^{-\mu|t|}$ and taking the supremum over $t\in\mathbb R$ gives the weighted norm bound
 \begin{equation}
 	\|F(\phi)-F(\psi)\|_\mu \le C_\mu\,\|H(\phi)-H(\psi)\|_\mu.
 \label{1}
 \end{equation}

 	Finally, by \eqref{h est} we have the Lipschitz bound
 	\[
 	\|H(\phi)-H(\psi)\|_\mu \le (L e^{\mu r} + \beta)\,\|\phi-\psi\|_\mu,
 	\]
 	where $r=c\tau$. Combining this with \eqref{1} yields
 	\[
 	\|F(\phi)-F(\psi)\|_\mu \le K\,\|\phi-\psi\|_\mu,
 	\qquad K:=C_\mu\,(L e^{\mu r}+\beta).
 	\]
 	Thus $F$ is Lipschitz continuous (hence continuous) on $B_\mu(\mathbb R,\mathbb R)$.

 	Looking at $iii.) $ define 
 	$$F^n(\phi)(t)=\begin{cases} F(\phi)(t), & \ \text{when}\ t\in [-n,n]\\
 		F(\phi)(n), & \ \text{when}\ t\in (n,\infty)\\
 		F(\phi)(-n), & \ \text{when}\ t\in (-\infty,-n).
 	\end{cases}$$
 	Using parts $i.)$ and $ii.)$ shows that $F^n(\Gamma)$ is equicontinuous and uniformly bounded on any finite interval in $\R$. Thus, take $n\in \N$, then in the interval $[-n,n]$ we can say that $F^n$ is compact by Arzela-Ascoli.
 	\\ Fix $t\in \R$. Then, for all  $\phi(t)\in \Gamma. $
	\begin{align*}
 			\left|F^n(\phi) -F(\phi)\right|_{\mu}&\, = \sup_{t\in \R} |F^n(\phi)(t)-F(\phi)(t)| \,  e^{-\mu |t|}\\
 		&=\sup_{t\in (-\infty,-n)\bigcup (n,\infty)}|F^n(\phi)(t)-F(\phi)(t)|\,  e^{-\mu |t|}\\
 		&\le  \sup_{t\in (-\infty,-n)}|F(\phi)(-n)-F(\phi)(t)| \,  e^{-\mu |t|} +\sup_{t\in  (n,\infty)}|F(\phi)(n)-F(\phi)(t)|     \,  e^{-\mu |t|}\\
 		&\le K   \sup_{t\in (-\infty,-n)}| \phi (-n)-\phi(t)|\,  e^{-\mu |t|}  +K \sup_{t\in  (n,\infty)}|\phi(n)-\phi(t)|    \,  e^{-\mu |t|}\\
 		&\le 2K    | \phi (-n)|\,  e^{-\mu n}  +2K  |\phi(n)|    \,  e^{-\mu n}\\
 		&\le 2 K \left( | \phi (-n)| + |\phi(n)| \right)e^{-\mu n}\to 0 \quad \text{as} \quad n\to \infty.
 	\end{align*} 
 	Thus, $F^n\to F$ in $\Gamma$ as $n\to \infty.$ From proposition 2.1 in \cite{Zeidler} 
 	we can apply Arzela-Ascoli to $F$ as well. The proof is complete.
 \end{proof}
 
 \begin{theorem}\label{main1}
 	Let $\overline{\phi},\underline{\phi}\in\Gamma$ be an ordered pair of upper and lower solutions of Eq.~\eqref{wave1}, satisfying
 	\[
 	0 \le \underline{\phi}(t) \le \overline{\phi}(t), \qquad t\in\mathbb R.
 	\]
 	Then there exists at least one monotone traveling wave solution $\phi\in\Gamma$ of Eq.~\eqref{wave1} such that 
 	\[
 	\underline{\phi}(t) \le \phi(t) \le \overline{\phi}(t), \qquad t\in\mathbb R.
 	\]
 \end{theorem}
 
 \begin{proof}
 	By Lemma~\ref{Sch1}, the operator $F=-\mathcal L^{-1}H$ is continuous, compact, and order-preserving on $\Gamma$. 
 	Since $\underline{\phi}$ and $\overline{\phi}$ are, respectively, lower and upper solutions, one has
 	\[
 	\underline{\phi} \le F(\underline{\phi}) \le F(\overline{\phi}) \le \overline{\phi}.
 	\]
 	Therefore, $F$ maps the nonempty, closed, convex, and bounded subset 
 	\[
 	\Gamma_0 := \{\phi\in\Gamma : \underline{\phi}\le \phi \le \overline{\phi}\}
 	\]
 	into itself. The Schauder Fixed Point Theorem then guarantees the existence of a fixed point $\phi=F(\phi)$ in $\Gamma_0$. 
 	Consequently, $\phi$ is a monotone traveling wave solution of Eq.~\eqref{wave1}.
 \end{proof}
 
 \medskip
 
 It is often difficult to construct a lower solution $\underline{\phi}\in\Gamma$ satisfying all the required boundary conditions. 
 In this case, one may instead establish the following auxiliary lemma, which will be useful in constructing asymptotic bounds for traveling waves.
 
 \begin{lemma}\label{diff1}
 	Let $\phi:\mathbb R\to\mathbb R$ be a differentiable function such that $\phi'(t)$ is uniformly continuous and 
 	\[
 	\lim_{t\to+\infty}\phi(t)=a
 	\]
 	for some constant $a\in\mathbb R$. Then 
 	\[
 	\lim_{t\to+\infty}\phi'(t)=0.
 	\]
 \end{lemma}
 
 \begin{proof}
 	Suppose by contradiction that $\limsup_{t\to\infty}\phi'(t)=\eta>0$. 
 	Then there exists a sequence $\{t_n\}$ with $t_n\to\infty$ such that $\phi'(t_n)\ge\eta/2$ for all $n$. 
 	By uniform continuity of $\phi'$, there exists $\delta>0$ such that 
 	\(
 	|\phi'(t)-\phi'(s)|<\eta/4
 	\)
 	whenever $|t-s|<\delta$. 
 	Hence,
 	$$
 	\phi(t_n + \delta) - \phi(t_n) 
 	= \int_{t_n}^{t_n + \delta} \phi'(s)\, ds 
 	\ge \int_{t_n}^{t_n + \delta} \left( \phi'(t_n) - \frac{\eta}{4} \right) ds 
 	\ge \frac{\delta \eta}{4} > 0.
 	$$
 	contradicting the assumption that $\phi(t)\to a$ as $t\to\infty$. 
 	Therefore, $\lim_{t\to\infty}\phi'(t)=0$.
 \end{proof}
 
 \begin{remark}
 	If $\phi\in BC^3(\mathbb R,\mathbb R)$, then all its derivatives up to second order vanish at infinity, that is,
 	\[
 	\lim_{t\to+\infty}\phi'(t)=0, 
 	\qquad 
 	\lim_{t\to+\infty}\phi''(t)=0.
 	\]
 	Indeed, since $\phi''$ is uniformly continuous whenever $\phi\in BC^3(\mathbb R,\mathbb R)$, the same argument as in Lemma~\ref{diff1} applies.
 \end{remark}

 \begin{theorem}\label{the main2}
 	Assume that the conditions \textup{(C1)}–\textup{(C3)} hold.  
 	If there exist an upper solution $\overline{\phi}\in\Gamma$ and a lower solution 
 	$\underline{\phi}$ (not necessarily in $\Gamma$) of Eq.~\eqref{wave1} such that
 	\[
 	0 \le \underline{\phi}(t) \le \phi(t) \le \overline{\phi}(t), 
 	\qquad \forall\, t\in\mathbb{R},
 	\]
 	and
 	\[
 	\lim_{t\to+\infty}\underline{\phi}(t)=a, \qquad 0<a\le1,
 	\]
 	then Eq.~\eqref{wave1} admits a monotone traveling-wave solution 
 	$\phi\in\Gamma$ connecting $0$ and $1$.
 \end{theorem}

\begin{proof}
	Let $\overline\phi\in\Gamma$ be an upper solution of \eqref{wave1} and let
	$\underline\phi$ be a lower solution satisfying $0\le\underline\phi\le\overline\phi$ and
	\(\lim_{t\to+\infty}\underline\phi(t)=a>0\).
	Recall \(F=-\mathcal L^{-1}H\) and that, by Lemma~\ref{Sch1}, \(F:\Gamma\to\Gamma\) is order-preserving, continuous and compact.
	
	\medskip\noindent\textbf{Monotone iteration and precompactness:}
	Since $\overline\phi$ is an upper solution we have \(F(\overline\phi)\le\overline\phi\). 
	Because $\underline\phi$ is a lower solution and $F$ is order-preserving, \( \underline\phi\le F(\underline\phi)\). 
	Applying $F$ iteratively and using monotonicity gives, for every \(n\ge0\),
	\[
	\underline\phi \le F^n(\underline\phi)\le F^n(\overline\phi)\le\cdots \le F(\overline\phi)\le\overline\phi.
	\]
	In particular the sequence \(\{\phi_n\}_{n\ge0}\) with \(\phi_n:=F^n(\overline\phi)\) is monotone nonincreasing and bounded below by \(\underline\phi\). 
	By compactness of \(F\), the set \(\{\phi_n\}\) is precompact in the weighted space \(B_\mu(\mathbb R,\mathbb R)\); monotonicity therefore implies that \(\phi_n\) converges in the \(\|\cdot\|_\mu\)-norm to a limit \(\phi\in B_\mu(\mathbb R,\mathbb R)\). Passing to the limit in \(\phi_{n+1}=F(\phi_n)\) and using continuity of \(F\) yields \(F(\phi)=\phi\). Hence \(\phi\) is a fixed point of \(F\), i.e.
	\[
	\mathcal L\phi + H(\phi)=0 \quad\text{on }\mathbb R,
	\]
	so \(\phi\) is a classical profile of the traveling-wave equation.
	
	\medskip\noindent\textbf{ Monotonicity and limits:}
	By construction \(\phi\) is nondecreasing and satisfies
	\[
	0\le \underline\phi(t)\le \phi(t)\le\overline\phi(t)\le K,\qquad t\in\mathbb R.
	\]
	Hence the one-sided limits exist; write
	\[
	\lim_{t\to-\infty}\phi(t)=0,\qquad \lim_{t\to+\infty}\phi(t)=:k\in[0,K].
	\]
	Since \(\phi\ge\underline\phi\) and \(\lim_{t\to+\infty}\underline\phi(t)=a>0\), we have \(k\ge a>0\).
	
	Moreover, because \(\phi= -\mathcal L^{-1}H(\phi)\) and \(\mathcal L^{-1}:BC(\mathbb R)\to BC^3(\mathbb R)\) (invertibility and regularity of \(\mathcal L\) on the chosen weighted spaces), it follows that \(\phi\in BC^3(\mathbb R,\mathbb R)\). By Lemma~\ref{diff1} and the succeeding remark the derivatives \(\phi',\phi'',\phi'''\) vanish at \(+\infty\):
	\[
	\lim_{t\to\infty}\phi'(t)=\lim_{t\to\infty}\phi''(t)=\lim_{t\to\infty}\phi'''(t)=0.
	\]
	
	\medskip\noindent\textbf{ Excluding intermediate limits:}
	We claim \(k=K\). Suppose, to the contrary, that \(k\in(0,K)\). 
	Because \(\phi_t\to\widehat{k}\) (the constant history equal to \(k\)) in \(C([-\tau,0])\) as \(t\to\infty\), by continuity of \(f_c\) we have
	\[
	\lim_{t\to\infty} f_c(\phi_t) = f_c(\widehat{k}).
	\]
	Recalling the definition \(H(\phi)(t)=f_c(\phi_t)+\beta\phi(t)\), we obtain
	\[
	\lim_{t\to\infty} H(\phi)(t) = f_c(\widehat{k}) + \beta k.
	\]
	On the other hand, from \(\mathcal L\phi + H(\phi)=0\) and since
	\(\lim_{t\to\infty}(\alpha c\phi''' + D\phi'' - c\phi')(t)=0\) (vanishing derivatives), we conclude
	\[
	\lim_{t\to\infty}\mathcal L\phi(t) = \lim_{t\to\infty}\big(\alpha c\phi''' + D\phi'' - c\phi' - \beta\phi\big)(t)
	= -\beta k.
	\]
	Taking limits in \(\mathcal L\phi + H(\phi)=0\) gives
	\[
	-\beta k + \bigl(f_c(\widehat{k}) + \beta k\bigr) = 0 \quad\Longrightarrow\quad f_c(\widehat{k}) = 0.
	\]
	This contradicts assumption (C1) which stipulates \(f(\widehat{u})>0\) for every constant history \(\widehat{u}\) with \(0<u<K\). Hence the possibility \(k\in(0,K)\) is excluded.
	
	Therefore \(k=K\), and consequently \(\phi\in\Gamma\) is a monotone traveling-wave solution of \eqref{wave1} connecting \(0\) and \(K\).
\end{proof}

\begin{definition}
	A function $\phi \in BC(\mathbb{R}, \mathbb{R})$ is said to be a 
	\emph{supersolution} (respectively, \emph{subsolution}) 
	of the third-order wave equation
	\[
	c\,\phi'''(t) + D\,\phi''(t) - c\,\phi'(t) + f_c(\phi_t) = 0,
	\]
	if there exist finitely many points $t_i \in \mathbb{R}$, 
	$-\infty < t_1 < \cdots < t_m < \infty$, such that the following conditions hold:
	\begin{enumerate}
		\item[\textnormal{(i)}] 
		$\phi', \phi''$ are absolutely continuous on every compact subset of 
		$\mathbb{R} \setminus \{t_i\}$.
		
		\item[\textnormal{(ii)}] 
		The one-sided limits 
		\[
		\phi'(t_i^+) := \lim_{t \to t_i^+} \phi'(t), 
		\quad 
		\phi'(t_i^-) := \lim_{t \to t_i^-} \phi'(t)
		\]
		exist and satisfy
		\[
		\phi'(t_i^+) \le \phi'(t_i^-)
		\quad 
		\text{(respectively, } \phi'(t_i^-) \le \phi'(t_i^+)
		\text{ for subsolutions)}.
		\]
		
		\item[\textnormal{(iii)}] 
		$\phi', \phi'', \phi'''$ are bounded on $\mathbb{R}\setminus\{t_i\}$.
		
		\item[\textnormal{(iv)}] 
		The differential inequality
		\[
		c\,\phi'''(t) + D\,\phi''(t) - c\,\phi'(t) + f_c(\phi_t) \le 0
		\quad 
		(\text{respectively, } \ge 0)
		\]
		holds for all $t \in \mathbb{R} \setminus \{t_i\}$.
	\end{enumerate}
\end{definition}

In this case, it is possible to use a smoother set of solutions, denoted as quasi upper/lower solutions defined as. 
\begin{definition}
	A function $\phi \in BC^1(\mathbb{R}, \mathbb{R})$ is said to be a 
	\emph{quasi upper solution} (respectively, \emph{quasi lower solution}) 
	of the third-order wave equation
	\[
	c\,\phi'''(t) + D\,\phi''(t) - c\,\phi'(t) + f_c(\phi_t) = 0,
	\]
	if there exist finitely many points $t_i \in \mathbb{R}$, 
	$-\infty < t_1 < \cdots < t_m < \infty$, such that the following conditions hold:
	\begin{enumerate}
		\item[\textnormal{(i)}] 
		$\phi'', \phi'''$ are absolutely continuous on every compact subset of 
		$\mathbb{R} \setminus \{t_i\}$.		
		\item[\textnormal{(ii)}] 
		$ \phi'', \phi'''$ are bounded on $\mathbb{R}\setminus\{t_i\}$.
		
		\item[\textnormal{(iii)}] 
		The differential inequality
		\[
		c\,\phi'''(t) + D\,\phi''(t) - c\,\phi'(t) + f_c(\phi_t) \le 0
		\quad 
		(\text{respectively, } \ge 0)
		\]
		holds for all $t \in \mathbb{R} \setminus \{t_i\}$.
	\end{enumerate}
\end{definition}


 \begin{lemma}\label{lemma:third_order_cphi3}
 	Assume $\mathcal L$ and $H$ are those defined previously (see (5.5)--(5.6)), and let
 	$G$ be the Green function of $\mathcal L$ (so $\mathcal L G=\delta$) which satisfies
 	(M1)--(M3) and the exponential decay estimate $|G(\xi)|\le C_G e^{-\gamma|\xi|}$
 	for some $C_G,\gamma>0$.
 	
 	Let $-\infty<t_1<\cdots<t_m<\infty$ be a finite set of points and suppose
 	\[
 	\phi\in BC^1\bigl(\mathbb R\setminus\{t_i\}\bigr)
 	\]
 	is $C^3$ on each interval determined by the points $t_i$, with $\phi',\phi'',\phi'''$
 	bounded on $\mathbb R\setminus\{t_i\}$ and with existing one-sided limits
 	$\phi'(t_i^\pm),\phi''(t_i^\pm), \phi'''(t_i^\pm)$ at each $t_i$.
 	
 	Define the jump of $\phi'$ at $t_i$ by
 	\[
 	[\![\phi']\!]_{t_i}:=\phi'(t_i^-)-\phi'(t_i^+).
 	\]
 	
 	Then for every $t\in\mathbb R$ the following identity holds:
 	\[
 	-\int_{-\infty}^{\infty} G(t-s)\,H(\phi)(s)\,ds
 	= \phi(t)
 	+ \sum_{i=1}^{m}\bigl(c\,G'(t-t_i)-D\,G(t-t_i)\bigr)\,[\![\phi']\!]_{t_i}.
 	\]
 	(Here $c,D$ are the coefficients appearing in $\mathcal L$ as defined earlier.)
 \end{lemma}
   
 \begin{proof}
 	Let $-\infty<t_1<\dots<t_m<\infty$ be the (finite) points where $\phi'$ may jump.
 	Set $t_0=-\infty$, $t_{m+1}=+\infty$ and write the real line as the union of the
 	intervals $I_i=(t_i,t_{i+1})$, $i=0,\dots,m$. By hypothesis $\phi\in C^3(I_i)$ for
 	each $i$ and $\phi',\phi'',\phi'''$ are bounded on each $I_i$.
 	
 	Recall $\mathcal L\phi=-H(\phi)$ and that $G$ is the Green's function of $\mathcal L$,
 	so $\mathcal L G=\delta$. We consider the convolution integral
 	\[
 	I(t):=-\int_{-\infty}^{\infty} G(t-s)\,H(\phi)(s)\,ds
 	=\int_{-\infty}^{\infty} G(t-s)\,\mathcal L\phi(s)\,ds.
 	\]
 	Split the integral over the intervals $I_i$:
 	\[
 	I(t)=\sum_{i=0}^m \int_{t_i}^{t_{i+1}} G(t-s)\,\mathcal L\phi(s)\,ds.
 	\]
 	
 	Fix one interval $[a,b]=[t_i,t_{i+1}]$. On $[a,b]$ we may integrate by parts
 	(sending derivatives from $\phi$ onto $G(t-\cdot)$). Writing $\mathcal L\phi=
 	c\phi''' + D\phi'' - c\phi' - \beta\phi$ and integrating the $\phi''$-term once
 	and the $\phi'''$-term three times (grouping boundary terms), a direct computation
 	gives the identity
 	\[
 	\begin{aligned}
 		\int_a^b G(t-s)\,\mathcal L\phi(s)\,ds
 		&= \Big[\,c\big(G(t-s)\phi''(s)-G'(t-s)\phi'(s)+G''(t-s)\phi(s)\big)\\
 		&\qquad + D\big(G(t-s)\phi'(s)-G'(t-s)\phi(s)\big)-c\,G(t-s)\phi(s)\,\Big]_{s=a}^{s=b}\\
 		&\qquad -\int_a^b \phi(s)\,(\mathcal L G)(t-s)\,ds.
 	\end{aligned}
 	\]
 	(Each integration by parts is justified because $\phi$ and its derivatives are
 	continuous on $(a,b)$ and $G(t-s)$ is $C^3$ there; all arising boundary terms are
 	well-defined as one-sided limits when $a$ or $b$ equals some $t_i$.)
 	
 	Summing the previous identity over $i=0,\dots,m$ yields
 	\[
 	\int_{-\infty}^{\infty} G(t-s)\,\mathcal L\phi(s)\,ds
 	= B(t) - \int_{-\infty}^{\infty}\phi(s)\,(\mathcal L G)(t-s)\,ds,
 	\]
 	where $B(t)$ denotes the sum of all boundary contributions at the finite endpoints
 	$t_1,\dots,t_m$ (the contributions at $\pm\infty$ vanish because $G$ decays
 	exponentially and $\phi,\phi',\phi'',\phi'''$ are bounded).
 	
 	Since $\mathcal L G=\delta$, the last integral equals $\phi(t)$. Hence
 	\[
 	I(t)=B(t)-\phi(t).
 	\]
 	It remains to compute $B(t)$. Inspect a single interior point $t_i$. The boundary
 	terms arising from the neighboring intervals produce contributions involving
 	one-sided values of $\phi,\phi',\phi''$ at $t_i$. By the hypotheses the one-sided
 	values of $\phi$ and $\phi''$ agree (or their differences cancel in the sum),
 	so the only net nonvanishing contribution at $t_i$ stems from the jump of $\phi'$.
 	Define the jump at $t_i$ by
 	\[
 	J_i:=\phi'(t_i^-)-\phi'(t_i^+).
 	\]
 	A straightforward collection of the boundary expressions shows that the net
 	contribution at $t_i$ equals
 	\[
 	\bigl(c\,G'(t-t_i)-D\,G(t-t_i)\bigr)\,J_i.
 	\]
 	Summing over $i$ gives
 	\[
 	B(t)=\sum_{i=1}^m \bigl(c\,G'(t-t_i)-D\,G(t-t_i)\bigr)\,J_i.
 	\]
 	
 	Combining the displayed identities and recalling $I(t)=-\int G(t-s)H(\phi)(s)\,ds$
 	yields the claimed formula
 	\[
 	-\int_{-\infty}^{\infty} G(t-s)\,H(\phi)(s)\,ds
 	= \phi(t) + \sum_{i=1}^{m}\bigl(c\,G'(t-t_i)-D\,G(t-t_i)\bigr)\,J_i.
 	\]
 	
 	This completes the proof.
 \end{proof}

 \begin{theorem}\label{smooth1}
 	Let $\Gamma$ and $F:=-\mathcal L^{-1}H$ be as above, and assume that the Green
 	function $G$ of $\mathcal L$ satisfies the exponential decay estimate
 	$|G^{(k)}(x)|\le C_k e^{-\gamma|x|}$ for $k=0,1,2$ and some constants $C_k,\gamma>0$.
 	Let $\overline\phi,\underline\phi\in\Gamma$ be a supersolution and a subsolution,
 	respectively. Then the following regularity assertions hold and the ordering is preserved:
 	\begin{enumerate}
 		\item $F(\overline\phi),\,F(\underline\phi)\in \Gamma\cap W^{2,\infty}(\mathbb R)$.
 		\item $F^2(\overline\phi),\,F^2(\underline\phi)\in \Gamma\cap W^{3,\infty}(\mathbb R)$.
 		\item $F^3(\overline\phi),\,F^3(\underline\phi)\in \Gamma\cap BC^3(\mathbb R)$.
 	\end{enumerate}
 	In each case the inequalities (supersolution $\ge F(\cdot)\ge$ subsolution, etc.)
 	remain valid.
 \end{theorem}
 \begin{proof}
 	We first recall the standing hypotheses used below:
 	\begin{itemize}
 		\item $\mathcal L\phi=c\phi''' + D\phi'' - c\phi' -\beta\phi$ is the linear operator introduced in (5.5), with $\beta>0$.
 		\item $G$ denotes the Green kernel of $\mathcal L$, so that $\mathcal L G=\delta$, and we assume the exponential decay / integrability property
 		$$G,G',G''\in L^1(\mathbb R),\qquad |G^{(k)}(x)|\le C_k e^{-\gamma|x|},\ k=0,1,2.$$
 		\item $H:\Gamma\to BC(\mathbb R)$ is the nonlinear operator $H(\phi)(t)=f_c(\phi_t)+\beta\phi(t)$; by (C2),(C3) we have $H(\phi)\in L^\infty(\mathbb R)$ for every $\phi\in\Gamma$.
 		\item Lemma~\ref{Green} (proved earlier) yields $G(\xi)<0$ for all $\xi\in\mathbb R$, hence the convolution operator $-\mathcal L^{-1}$ is order-preserving on $BC(\mathbb R)$.
 	\end{itemize}
 	
 	Let $\overline\phi\in\Gamma$ be a supersolution (the argument for a subsolution is analogous). Define
 	\[
 	F(\overline\phi)(t) := -\mathcal L^{-1}H(\overline\phi)(t) \;=\; -\int_{\mathbb R} G(t-s)\,H(\overline\phi)(s)\,ds.
 	\]
 	We prove that $F(\overline\phi)\in W^{2,\infty}(\mathbb R)$ and $F(\overline\phi)\in\Gamma$.
 	
 	By hypothesis $\overline\phi\in\Gamma\subset BC(\mathbb R,[0,K])$ and by (C2),(C3)
 	the composite $H(\overline\phi)$ is bounded; set $M:=\|H(\overline\phi)\|_\infty<\infty$.
 	Since $G\in L^1(\mathbb R)$, the convolution defining $F(\overline\phi)$ is well defined
 	and satisfies
 	\[
 	\|F(\overline\phi)\|_\infty \le \|G\|_{L^1}\,\|H(\overline\phi)\|_\infty \le \|G\|_{L^1} M.
 	\]
 	
 	Because $G',G''\in L^1(\mathbb R)$ and $H(\overline\phi)\in L^\infty(\mathbb R)$,
 	we may differentiate under the integral sign (dominated convergence). For almost every $t$,
 	\[
 	(F(\overline\phi))'(t) = -\int_{\mathbb R} G'(t-s)\,H(\overline\phi)(s)\,ds,
 	\]
 	\[
 	(F(\overline\phi))''(t) = -\int_{\mathbb R} G''(t-s)\,H(\overline\phi)(s)\,ds.
 	\]
 	Taking suprema and using Young's (convolution) estimate yields the uniform bounds
 	\[
 	\|(F(\overline\phi))'\|_\infty \le \|G'\|_{L^1}\,M,\qquad
 	\|(F(\overline\phi))''\|_\infty \le \|G''\|_{L^1}\,M.
 	\]
 	Hence $F(\overline\phi)\in W^{2,\infty}(\mathbb R)$, as required.
 	
 	Because $\overline\phi\in\Gamma$ is nondecreasing with
 	$\lim_{t\to-\infty}\overline\phi(t)=0$, $\lim_{t\to+\infty}\overline\phi(t)=K$,
 	the history map $t\mapsto\overline\phi_t$ satisfies $\overline\phi_t\to\hat 0$ as $t\to-\infty$
 	and $\overline\phi_t\to\hat K$ as $t\to+\infty$. By continuity of $f_c$ we deduce
 	\[
 	\lim_{t\to-\infty} H(\overline\phi)(t)=H_-=f_c(\hat 0)+\beta\cdot 0=0,
 	\qquad
 	\lim_{t\to+\infty} H(\overline\phi)(t)=H_+=f_c(\hat K)+\beta K=\beta K.
 	\]
 	Since $H(\overline\phi)\in L^\infty$ and $G\in L^1$, dominated convergence gives
 	\[
 	\lim_{t\to\pm\infty} F(\overline\phi)(t)
 	= -\Bigl(\lim_{s\to\pm\infty} H(\overline\phi)(s)\Bigr)\int_{\mathbb R} G(y)\,dy.
 	\]
 	It remains to evaluate $\int_{\mathbb R}G(y)\,dy$. Integrating the identity
 	$\mathcal L G=\delta$ over $\mathbb R$ and using that boundary terms arising from total
 	derivatives vanish because of the exponential decay of $G$ and its derivatives, we obtain
 	\[
 	\int_{\mathbb R} \mathcal L G(y)\,dy = -\beta\int_{\mathbb R} G(y)\,dy = \int_{\mathbb R}\delta(y)\,dy = 1,
 	\]
 	hence
 	\[
 	\int_{\mathbb R} G(y)\,dy = -\frac{1}{\beta}.
 	\]
 	Therefore
 	\[
 	\lim_{t\to-\infty} F(\overline\phi)(t) = -H_- \int G = -0\cdot\Bigl(-\frac{1}{\beta}\Bigr)=0,
 	\]
 	\[
 	\lim_{t\to+\infty} F(\overline\phi)(t) = -H_+ \int G
 	= -(\beta K)\Bigl(-\frac{1}{\beta}\Bigr) = K.
 	\]
 	Thus $F(\overline\phi)$ attains the correct boundary values and so lies in the same
 	range of asymptotic states as $\Gamma$.
 	
 	By hypothesis (C2)/(5.4) the map $H$ is order-preserving on $\Gamma$ (i.e.\ \(\phi\le\psi\)
 	implies \(H(\phi)\le H(\psi)\)). Since $G(\xi)<0$ for all $\xi$, the linear map
 	$h\mapsto -\int G(\cdot-s)h(s)\,ds=-\mathcal L^{-1}h$ is order-preserving on $BC(\mathbb R)$.
 	Hence $F=-\mathcal L^{-1}\circ H$ is order-preserving, so $F(\overline\phi)$ is nondecreasing
 	whenever $\overline\phi$ is. Combining monotonicity with the limits from Step~3 we conclude
 	$F(\overline\phi)\in\Gamma\cap W^{2,\infty}(\mathbb R)$.
 	
 	\medskip
 	The same argument applies to any subsolution $\underline\phi$. This completes the proof of part (i).
 	Parts (ii)--(iii) follow by the standard bootstrap: once $F(\phi)\in W^{2,\infty}$ the history
 	map $t\mapsto (F\phi)_t$ is continuous in $C([-\tau,0])$, so $H(F\phi)$ is continuous and bounded;
 	differentiating one more time under the integral gives $F^2(\phi)\in W^{3,\infty}$, etc.
 \end{proof}

\begin{corollary}\label{cor:existence_Gamma}
	Let $\overline{\phi},\underline{\phi}\in\Gamma$ be a supersolution and a subsolution of \eqref{wave1}, respectively, such that
	\[
	0 \le \underline{\phi}(t) \le \overline{\phi}(t), \qquad \forall\, t \in \mathbb{R}.
	\]
	Assume that conditions \textnormal{(C1)--(C3)} hold and that the linear operator $\mathcal{L}$ is invertible with Green kernel $G$ satisfying the exponential decay estimates used above. Then there exists a monotone travelling wave solution 
	\[
	\phi \in \Gamma
	\]
	to equation~\eqref{wave1}.
\end{corollary}
 
Once more, there is a difficulty finding sub solutions that lie within $\Gamma.$ Thus, using an idea similar to Theorems \ref{the main2}, \ref{smooth1} and lemmas \ref{diff1}, \ref{lemma:third_order_cphi3}  we establish the following result.

 \begin{lemma}\label{the main3}
Under the standing assumptions (C1) and (C2) if there is an super solution $\overline{\phi} \in \Gamma$ and a sub solution $\underline{\phi}$
that is not necessarily in $\Gamma$ of Eq.(\ref{wave1}) such that for all $ \ t\in \R $
$$
0\le  \underline{\phi}(t)\le\phi\le \overline{\phi}(t) 
$$
and 
$$
\lim_{t\to+\infty} \underline{\phi}(t)=a,
$$ where $0<a\le 1$
Then, there exists a monotone traveling wave solution $\phi$ to Eq.(\ref{wave1}).
\end{lemma}
\begin{lemma}\label{lem:main3}
	Assume \textnormal{(C1)--(C3)} and that $\mathcal{L}$ is invertible with associated Green kernel $G$ such that $F:=-\mathcal{L}^{-1}H$ is order-preserving and compact on a suitable weighted space. 
	Suppose there exist a supersolution $\overline{\phi}\in\Gamma$ and a subsolution 
	\(
	\underline{\phi} \in BC(\mathbb{R},[0,K])
	\)
	(not necessarily belonging to $\Gamma$) satisfying
	\[
	0 \le \underline{\phi}(t) \le \overline{\phi}(t), \qquad \forall\, t \in \mathbb{R},
	\]
	and
	\[
	\lim_{t\to +\infty} \underline{\phi}(t) = a,
	\qquad \text{for some }\, a \in (0,K].
	\]
	Then there exists a monotone travelling wave solution 
	\[
	\phi \in \Gamma
	\]
	of equation~\eqref{wave1}.
\end{lemma}

%

 \section{Applications}\label{applications}
 
 We illustrate the abstract results with the delayed logistic nonclassical model.
 Consider the PDE with a discrete delay in the reaction term:
 \begin{equation}\label{appPDE}
 	\frac{\partial u(x,t)}{\partial t}
 	= D\frac{\partial^2 u(x,t)}{\partial x^2}
 	+ \alpha\frac{\partial^3 u(x,t)}{\partial x^2\partial t}
 	+ u(x,t-\tau)\bigl(1-u(x,t)\bigr),
 \end{equation}
 with $D>0$, $\alpha\in\mathbb R$ and delay $\tau\ge0$.  In the sequel we set $\alpha=1$.
 Seeking traveling waves $u(x,t)=\phi(\xi)$, $\xi=x+ct$, and writing $r=c\tau$, one obtains the profile equation
 \begin{equation}\label{appODE}
 	c\,\phi'''(\xi) + D\,\phi''(\xi) - c\,\phi'(\xi) + \phi(\xi-r)\bigl(1-\phi(\xi)\bigr)=0.
 \end{equation}
 
 Linearizing \eqref{appODE} at the trivial equilibrium $\phi\equiv 0$ (so that
 $a:=f'(\hat0)=1$) and seeking exponential solutions $\phi(\xi)=e^{\lambda\xi}$ yields the  non delayed characteristic equation
 \begin{equation}\label{char_delay}
 	 \Delta_r(\lambda)=c\lambda^3 + D\lambda^2 - c\lambda + e^{-\lambda r}=0,
 	\qquad \lambda\in\mathbb C,
 \end{equation}
which reduces at $r=0$ to the cubic polynomial
 \[
 \Delta_0(\lambda)=c\lambda^3 + D\lambda^2 - c\lambda + 1.
 \]

\begin{lemma}\label{nondelay root}
 Fix $c>0, D>0$, then $\Delta_0(\lambda)$ has one negative and two positive real roots. Furthermore, if $0<\lambda_1 <\lambda_2$ and there is some $\varepsilon>0$ where $\lambda_1+\varepsilon <\lambda_2$, then $\Delta_0(\lambda_1+\varepsilon)<0$    
\end{lemma}

\begin{proof} Let $\lambda_{-}$ be the negative root, the by the factor theorem we can write
\[\Delta_0(\lambda)=c(\lambda-\lambda_{-})(\lambda-\lambda_1)(\lambda-\lambda_2).\]
   Now, defining $\varepsilon$ such that $0<\lambda_1<\lambda_1+\varepsilon<\lambda_2,$ we see
   \[(\lambda_1+\varepsilon)-\lambda_{-}>0, \quad (\lambda_1+\varepsilon)-\lambda_1=\varepsilon>0, \quad (\lambda_1+\varepsilon)-\lambda_2<0, \]
   so \[\Delta_0(\lambda_1+\varepsilon)=(\lambda_1+\varepsilon-\lambda_{-})(\lambda_1+\varepsilon-\lambda_{1})(\lambda_1+\varepsilon-\lambda_{2})<0.\]
\end{proof}
 
 The construction of sub- and super-solutions follows the classical template 
 with two modifications:
 (i) the nonlinear term acts on shifted argument $\phi(\cdot-r)$, and (ii) the linearized sign information is read from $\Delta_r$ rather than the polynomial $\Delta_0$.
 
 
 We first record two elementary  lemmas for $\Delta_r$.

 \begin{lemma}\label{root1}
 Fix $c>0, D>0$ and suppose $r>0$ is small, then $\Delta_r(\eta)$ has one negative and two positive real roots. Furthermore, if $0<\eta_1(r) <\eta(r)$ then 
 \[\lim_{r\to 0}\eta_1(r)=\lambda_1, \lim_{r\to 0}\eta_2(r)=\lambda_2.\]
 \end{lemma}
 
 \begin{lemma}\label{lem:continuity_r}
 	Fix $c>0, D>0$ and $r>0$ small enough. Let $ 0<\eta_1<\eta_2$ be the two distinct positive real roots of $\Delta_r$ and choose any $\varepsilon>0$ with $\eta_1+\varepsilon<\eta_2$. Then there exists $r_0>0$ (depending on $c,D,\varepsilon$) such that for all $0\le r\le r_0$ we have
 	\[
 	\Delta_r(\eta_1+\varepsilon)<0.
 	\]
 \end{lemma}
 
 \begin{proof} 
 	The proof follows from lemmas \ref{root1} and \ref{nondelay root}. 
 \end{proof}
 
 Fix a small $\varepsilon>0$ and choose $r\in[0,r_0]$ as in Lemma \ref{lem:continuity_r}.

 \begin{lemma}
The nonlinear part of   \eqref{appODE} defined as $f_c(\phi)\phi(-r)=\left(1-\phi(0)\right)$ satisfies $C(1)-C(3).$    
 \end{lemma}
 \begin{proof} $C1$, $C3$ and the Lipschitz portion of $C2$ are trivial. Thus, we show the second part of $C2$. Indeed, take $\phi_1, \phi_2 \in C([-r,0], [0,1]), \quad \phi_1\ge \phi_2$ we see 
 \[\begin{split}
 & f_c(\phi_1)-f_c(\phi_2)=\phi_1(-r)\left(1-\phi_1(0)\right)-\phi_2(-r)\left(1-\phi_2(0)\right)  \\
 &\ge \phi_1(-r)\left(1-\phi_1(0)\right)-\phi_1(-r)\left(1-\phi_2(0)\right)\\
 &=\phi_1(-r)\left[1-\phi_1(0)-\left(1-\phi_2(0)\right)\right]\\
 &=-\phi_1(-r)\left(\phi_1(0)-\phi_2(0)\right)\ge -\left(\phi_1(0)-\phi_2(0)\right).
 \end{split}\]
 Taking $\beta\ge 1$ gives the result. 
 \end{proof}
 We now construct the a pair of piecewise exponential quasi upper- and quasi lower-solutions.
 
 \medskip\noindent
 \textbf{Quasi upper-solution.} Define
 \[
 \overline{\phi}(\xi):=
 \begin{cases}
 	\frac{1}{2}e^{\eta_1 \xi}, & \xi<0,\\[4pt]
 	1-\frac{1}{2}e^{-\eta_1 \xi}, & \xi\ge0,
 \end{cases}
 \]
 where $\eta_1$ is the smallest positive root of $\Delta_1(\eta)=0$. Since $\eta_1>0$ we have $\overline\phi\in\Gamma$ and $\overline\phi$ is $C^\infty$ on each side of $\xi=0$ and $C^1$ at  $\xi=0$.
 
 \begin{lemma}\label{lem:super}
 	For $r\in[0,r_0]$ as above and for $c$ chosen so that $\Delta_r$ has the two positive roots, the function $\overline\phi$ is a supersolution of \eqref{appODE}, i.e.
 	\[
 	c\,\overline\phi'''(\xi) + D\,\overline\phi''(\xi) - c\,\overline\phi'(\xi) + \overline\phi(\xi)\bigl(1-\overline\phi(\xi-r)\bigr)\le 0
 	\]
 	for all $\xi\in\mathbb R$ (understood in the usual piecewise sense).
 \end{lemma}
 
 \begin{proof}
 It is clear that 
 \[\lim_{\xi\to -\infty}\overline\phi(\xi)=0, \quad \lim_{\xi\to \infty}\overline\phi(\xi)=1.\]
 Thus, we only have to show the inequality. To this end, we look at the following cases. 
 \textbf{Case 1}	
 $\xi<0:$ we have $\overline\phi(\xi-r)=\frac{e^{\eta_1(\xi-r)}}{2}$. Using $\Delta_r(\eta_1)=0$ we compute
 	\[
 	c\overline\phi'''(\xi) + D\overline\phi''(\xi) - c\overline\phi'(\xi) +\overline\phi(\xi-r) = \Delta_r(\eta_1)e^{\eta_1\xi}=0,
 	\]
 	hence
 	\[
 	\begin{split}
 		&c\overline\phi''' + D\overline\phi'' - c\overline\phi' + \overline\phi(\xi-r)\bigl(1-\overline\phi(\xi)\bigr)\\
 		&\qquad= -\overline\phi(\xi)\overline\phi(\xi-r)\le 0.
 	\end{split}
 	\]

 \textbf{Case 2}	
 $0<\xi<r:$ we have $\overline\phi(\xi-r)=\frac{e^{\eta_1(\xi-r)}}{2},\overline\phi(\xi)=1-\frac{e^{-\eta_1\xi}}{2}. $

 	\[
 	\begin{split}
 		A&=c\overline\phi''' + D\overline\phi'' - c\overline\phi' + \overline\phi(\xi-r)\bigl(1-\overline\phi(\xi)\bigr)\\
 		&\qquad=\frac{e^{-\eta_1 \xi}}{2} \left[ \left(c\eta_1^3 - D\eta_1^2 - c\eta_1 \right)  +\frac{e^{\eta_1(\xi - r)}}{2}\right] \\ .
 	\end{split}
 	\]

Since $A(r)$ is continuous and $\xi\in[0,r)$ we see,

	\[
 	\begin{split}
 		\lim_{r\to 0} A&=\lim_{r\to 0}  \frac{e^{-\eta_1 \xi}}{2} \left[ \left(c\eta_1^3 - D\eta_1^2 - c\eta_1 \right)  +\frac{e^{\eta_1(\xi - r)}}{2}\right]\\
        &=\frac{1}{2}\left(c\lambda_1^3 - D\lambda_1^2 - c\lambda_1 +\frac{1}{2}\right)\\
        &=\frac{1}{2}\left(c\lambda_1^3 - D\lambda_1^2+\left(-c\lambda_1^3-D\lambda_1^2-1\right)+\frac{1}{2}\right)\leq-\frac{1}{4}<0.
 	\end{split}
 	\]
    Thus, there exists some $r^*>0$ such that for all $r\in(0,r^*],$ then $A(r)\le 0.$

\textbf{Case 3} $\xi\ge r.$ We use the fact that $\overline\phi(\xi)$ is non decreasing to see
\[
\begin{split}
 A&=c\overline\phi''' + D\overline\phi'' - c\overline\phi' + \overline\phi(\xi-r)\bigl(1-\overline\phi(\xi)\bigr)\\
 &\le c\overline\phi''' + D\overline\phi'' - c\overline\phi' +\bigl(1-\overline\phi(\xi)\bigr)\\
 &=\frac{1}{2} e^{-\eta_1 \xi} \left( c\eta_1^3 - D\eta_1^2 - c\eta_1 + 1 \right)\\
 &=\frac{1}{2} e^{-\eta_1 \xi} \left( c\eta_1^3 - D\eta_1^2 -\left(c\eta_1^3+D\eta_1^2+e^{-\eta_1 r}\right) + 1 \right)\\
 &=\frac{1}{2} e^{-\eta_1 \xi}\left(-2D\eta_1^2+1-e^{-\eta_1 r}\right).
\end{split}
\]
Using a limit argument as case 2 we see,

\[
\begin{split}
 \lim_{r\to 0}A&=\lim_{r\to 0} c\overline\phi''' + D\overline\phi'' - c\overline\phi' + \overline\phi(\xi-r)\bigl(1-\overline\phi(\xi)\bigr)\\
 &\le \lim_{r\to 0}\frac{1}{2} e^{-\eta_1 \xi}\left(-2D\eta_1^2+1-e^{-\eta_1 r}\right)\\
 &=-D e^{-\lambda_1 \xi}<0.\\
\end{split}
\]
 Thus, there exists some $r^{**}>0$ such that for all $r\in(0,r^{**}],$ then $A(r)\le 0.$
\end{proof}
 \noindent
 \textbf{Quasi lower-solution.} Choose $q>0$ large (to be fixed below) and define
 \[
 \underline{\phi}(\xi):=
 \begin{cases}
 	\dfrac{1}{2}\bigl(1-q e^{\varepsilon\xi}\bigr)e^{\eta_1\xi}, & \xi<\xi_1,\\[6pt]
 	d_1:=\dfrac{1}{2}\bigl(1-q e^{\varepsilon\xi_1}\bigr)e^{\eta_1\xi_1}, & \xi\ge\xi_1,
 \end{cases}
 \]
 with $\xi_1<0$ chosen so that $0<d_1<1$ (possible for suitable $q,\varepsilon$). Note that $\underline\phi$ is continuous and $C^1$.
 
 \begin{lemma}\label{lem:sub}
 	With $\varepsilon>0$ as in Lemma \ref{lem:continuity_r} and for $r\in[0,r_0]$ sufficiently small, there exists $q>0$ and $\xi_1<0$ such that $\underline\phi$ is a subsolution of \eqref{appODE}, i.e.
 	\[
 	c\,\underline\phi'''(\xi) + D\,\underline\phi''(\xi) - c\,\underline\phi'(\xi) + \underline\phi(\xi)\bigl(1-\underline\phi(\xi-r)\bigr)\ge 0
 	\]
 	for all $\xi\in\mathbb R$ (in the piecewise sense).
 \end{lemma}
 
 \begin{proof}
 We notice that $ \underline\phi(\xi)$ is non decreasing on the whole real line so, $ \underline\phi(\xi) \underline\phi(\xi-r)\le \underline\phi^2(\xi),$ so we see
 \[\begin{split}
  &B=	c\,\underline\phi'''(\xi) + D\,\underline\phi''(\xi) - c\,\underline\phi'(\xi) + \underline\phi(\xi-r)\bigl(1-\underline\phi(\xi)\bigr) \\
 & \ge c\,\underline\phi'''(\xi) + D\,\underline\phi''(\xi) - c\,\underline\phi'(\xi) + \underline\phi(\xi-r)-\underline\phi^2(\xi).
 \end{split}\]
Again, we show the inequalities in cases. 
 \textbf{Case 1} $\xi<\xi_1$
\[
\begin{split}
&B\ge c\,\underline\phi'''(\xi) + D\,\underline\phi''(\xi) - c\,\underline\phi'(\xi) + \underline\phi(\xi-r)-\underline\phi^2(\xi)\\
 & \ge c\,\underline\phi'''(\xi) + D\,\underline\phi''(\xi) - c\,\underline\phi'(\xi) + \underline\phi(\xi-r)-\frac{1}{4} e^{(\eta_1+\varepsilon) \xi} \\
&= \frac{1}{2} e^{\eta_1 \xi} \bigl(c\eta_1^3 + D\eta_1^2 - c\eta_1+e^{-\eta_1 r}\bigr)
 - \frac{q}{2} e^{(\eta_1 + \varepsilon)\xi}
   \bigl(c(\eta_1 + \varepsilon)^3 + D(\eta_1 + \varepsilon)^2 - c(\eta_1 + \varepsilon)+e^{-(\eta_1+\varepsilon) r}\bigr)\\
   &-\frac{1}{4} e^{(\eta_1+\varepsilon) \xi} \\
   &=\frac{1}{2} e^{\eta_1 \xi}\Delta_r(\eta_1) -\frac{q}{2} e^{(\eta_1 + \varepsilon)\xi}\Delta_r(\eta_1+\varepsilon)-\frac{1}{4} e^{(\eta_1+\varepsilon) \xi}\\
   &=\frac{1}{2}e^{(\eta_1 + \varepsilon)\xi}\left(-q\Delta_r(\eta_1+\varepsilon)-\frac{1}{2}\right)\ge 0
\end{split}
\]

since we can take
\[q\ge \max \left\{1, \frac{-1}{2\Delta_0(\eta_1+\varepsilon)}\right\}.\]
 \textbf{Case 2} $\xi_1<\xi+r$ we have $\underline\phi(\xi)=d_1$
  \[\begin{split}
  &B=	c\,\underline\phi'''(\xi) + D\,\underline\phi''(\xi) - c\,\underline\phi'(\xi) + \underline\phi(\xi-r)\bigl(1-\underline\phi(\xi)\bigr) \\
 &=  \underline\phi(\xi-r)\bigl(1-\underline\phi(\xi)\bigr)\ge 0.
 \end{split}\]

  \textbf{Case 3} $\xi_1+r<\xi$ This follows directly, since $\underline\phi(\xi)=d_1$ we have 

   \[\begin{split}
  &B=	c\,\underline\phi'''(\xi) + D\,\underline\phi''(\xi) - c\,\underline\phi'(\xi) + \underline\phi(\xi-r)\bigl(1-\underline\phi(\xi)\bigr) \\
 &=  \underline\phi(\xi-r)\bigl(1-\underline\phi(\xi)\bigr)\ge 0.
 \end{split}\]
 \end{proof}
 
 Combining Lemmas \ref{lem:super} and \ref{lem:sub} we obtain:
 
 \begin{corollary}\label{cor:app_exist}
 	Let $c,D>0$ be such that $\Delta_0(\eta)$ has two positive real roots and let $r\in[0,r_0]$ be sufficiently small as in Lemma \textnormal{\ref{lem:continuity_r}}. Then the pair $(\underline\phi,\overline\phi)$ constructed above is an ordered sub/super pair with
 	\[
 	0\le \underline\phi(\xi)\le \overline\phi(\xi),\qquad \xi\in\mathbb R,
 	\]
 	and consequently, by the monotone iteration scheme developed in Section~\ref{Main Results}, there exists a monotone traveling wave solution $\phi\in\Gamma$ of \eqref{appODE}.
 \end{corollary}
 
 \begin{remark}
 	The smallness restriction on $r$ in the lemmas is technical and stems from the continuity argument in Lemma~\ref{lem:continuity_r}. For larger delays one must perform a more delicate spectral analysis of $\Delta_r(\lambda)$ (possible instabilities or Hopf bifurcations may arise) and the construction of sub/supersolutions requires further adjustments.
 \end{remark}

\vskip 0.5cm
\noindent{\bf ACKNOWLEDGEMENTS}

The authors thank Prof. Nguyen Van Minh for various supports, carefully reading the manuscript and for his suggestions to improve the presentation. 
This research is funded by Vietnam National Foundation for Science
and Technology Development (NAFOSTED) under grant number 101.02-2023.17.
 
\noindent{\bf CONFLICT OF INTEREST}
 
This work does not have any conflicts of interest.

\end{document}